\documentclass[a4paper]{article}%
\usepackage{amsmath}
\usepackage{amsfonts}
\usepackage{amssymb}
\usepackage{graphicx}%
\providecommand{\U}[1]{\protect\rule{.1in}{.1in}}
\newtheorem{theorem}{Theorem}

\newtheorem{definition}[theorem]{Definition}

\newtheorem{lemma}[theorem]{Lemma}

\newenvironment{proof}[1][Proof]{\noindent\textbf{#1.} }{\ \rule{0.5em}{0.5em}}

\begin{document}
\title{Global well-posedness of the Toda lattice on an exact spectral phase space}
\author{Shuo Zhang$^{1}$ }
\date{}
\maketitle

\begin{abstract}
	We identify an exact spectral phase space for the two-sided Toda lattice.
	Let $q=\{a_n,b_n\}_{n\in\mathbb Z}$ be coefficients of the right and left
	half-line Jacobi operators and denote their spectral measures by
	$\sigma_{\pm}^{q}$. Define a phase space
	\[
	\mathcal Q=\left\{
	\begin{array}
		[c]{c}%
		q=\{a_n,b_n\}_{n\in\mathbb Z}:
		a_n>0,\ b_{n} \in \mathbb{R} \text{ and} \\
		\int_{\mathbb R}e^{c|\lambda|}\sigma^q_\pm(d\lambda)<\infty
		\text{ for every }c>0
	\end{array}
	\right\} .
	\]
	The integrability condition makes the representing measures unique. We prove
	that $q\in\mathcal Q$ if and only if the Toda lattice with initial datum $q$
	admits a classical solution for all positive and negative times. Moreover,
	the solution remains in $\mathcal Q$, is unique, and depends continuously on
	the initial datum, uniformly on compact time intervals.
\end{abstract}

\section{Introduction}
In 1967 M. Toda \cite{MT} introduced\footnotetext[1]{School of Math. Nanjing Univ. Nanjing China
	shuozhang@smail.nju.edu.cn} an anharmonic system
described by an infinite dimensional system of equations%
\[%
\begin{array}
	[c]{l}%
	\overset{\cdot}{p}_{n}(t)=e^{-\left(  q_{n}(t)-q_{n-1}(t)\right)
	}-e^{-\left(  q_{n+1}(t)-q_{n}(t)\right)  }\\
	\overset{\cdot}{q}_{n}(t)=p_{n}(t)
\end{array}
\]
with $n\in\mathbb{Z}$. This equation is called the Toda lattice, and it is
rewritten in an equivalent form%
\begin{equation}%
	\begin{array}
		[c]{l}%
		\overset{\cdot}{a}_{n}(t)=a_{n}(t)\left(  b_{n+1}(t)-b_{n}(t)\right) \\
		\overset{\cdot}{b}_{n}(t)=2\left(  a_{n}(t)^{2}-a_{n-1}(t)^{2}\right)
	\end{array}
	\label{1.1}%
\end{equation}
by Flaschka variables \cite{HF1}\cite{HF2}
\[
a_{n}(t)=\dfrac{1}{2}e^{-\left(  q_{n}(t)-q_{n-1}(t)\right)  /2}\text{,
	\ \ }b_{n}(t)=-\dfrac{1}{2}p_{n-1}(t)\text{.}%
\]
A \textbf{doubly global classical solution} to Toda lattice means that every coordinate $a_n,b_n$ belongs to $C^1(\mathbb{R})$ and
\begin{equation}
	a_n(t)>0,\qquad b_n(t)\in\mathbb{R},
	\qquad (n,t)\in\mathbb{Z}\times\mathbb{R}.
	\label{1.2}
\end{equation}
It is known that this equation has
infinitely many invariants and has a unique solution for any bounded initial
data (see \cite{GT}).

Several works treat \eqref{1.1} for unbounded initial data. Ifantis and
Vlachou \cite{EI} considered the half-line $\mathbb Z_+$ and obtained an
explicit representation of the solution. Let $J_+^q$ be the Jacobi expression
with coefficients $q=\{a_n,b_n\}_{n\in\mathbb Z_+}$,
\[
\begin{cases}
	(J_+^qu)_n=a_nu_{n+1}+a_{n-1}u_{n-1}+b_nu_n,
	& n\geq1,\\
	(J_+^qu)_0=a_0u_1+b_0u_0,
	& n=0.
\end{cases}
\]
and let $\sigma_+$ be its boundary spectral measure, so that%
\[
\left\langle(J_+^q-z)^{-1}\delta_0,\delta_0\right\rangle
=\int_{\mathbb R}\frac{\sigma_+(\mathrm d\lambda)}{\lambda-z}.
\]
Assuming that $\int_{\mathbb R}e^{t\lambda}\sigma_+(\mathrm d\lambda)<\infty$
for every $t\in\mathbb R$, they showed that the spectral measure at time $t$
is
\[
\frac{e^{2t\lambda}\sigma_+(\mathrm d\lambda)}
{\int_{\mathbb R}e^{2t\lambda}\sigma_+(\mathrm d\lambda)}.
\]
The coefficients are then recovered from the inverse spectral problem.

For the full lattice, Aggarwal \cite[Proposition~4.7]{AA} proved global
existence under the growth condition
\[
a_{n}>0\text{, }b_{n}\in\mathbb{R}\text{ and }a_{n}+\left\vert b_{n}%
\right\vert =O\left(  \left\vert n\right\vert ^{\alpha}\right)
\qquad (n\to\pm\infty)
\]
for some $\alpha\in[0,1)$. The proof proceeds by finite-volume approximation.

We establish global well-posedness on an exact phase space characterized by spectral measures. Specifically, we define the spectral phase space as follows:
\[
\mathcal Q=\left\{
\begin{array}
	[c]{c}%
	q=\{a_n,b_n\}_{n\in\mathbb Z}:
	a_n>0,\ b_{n} \in \mathbb{R} \text{ and} \\
	\int_{\mathbb R}e^{c|\lambda|}\sigma^q_\pm(d\lambda)<\infty
	\text{ for every }c>0
\end{array}
\right\} .
\]
where $\sigma^q_{\pm}$ are the boundary spectral measures of the right and
left half-line Jacobi operators associated with $q$. The exponential
integrability in the definition implies that both half-line Jacobi expressions
are essentially self-adjoint, so their spectral measures are uniquely defined.
Let
\[
\mathcal P_{\exp}(\mathbb{R}) :=\left\{\mu\ge0:\ \mu(\mathbb{R})=1,\quad \int_{\mathbb{R}}e^{c|x|}\mu(d x)<\infty\ \text{for every }c>0\right\}.
\]
For $\mu\in\mathcal P_{\exp}(\mathbb{R})$, define its bilateral Laplace transform by
\begin{equation}
	\mathcal{L}_\mu(z)=\int_{\mathbb{R}}e^{xz}\mu(d x),\qquad z\in\mathbb{C}.
	\label{1.3}
\end{equation}
It is entire.  Put
\begin{equation}
	p_R(f):=\sup_{|z|\le R}|f(z)|,
	\qquad
	\vartheta(s):=\frac{s}{1+s},\quad s\ge0,
	\label{1.4}
\end{equation}
and define
\begin{equation}
	d_{\mathcal{L}}(\mu,\nu)
	:=\sum_{m=1}^{\infty}2^{-m}
	\vartheta\!\left(p_m(\mathcal{L}_\mu-\mathcal{L}_\nu)\right).
	\label{1.5}
\end{equation}
The metric on $\mathcal{Q}$ is defined by
\begin{equation}
	\mathrm{d}\left(  q_{1},q_{2}\right) = \left| a_{-1}(q_{1}) - a_{-1}(q_{2}) \right| +  d_{\mathcal{L}}(\sigma_{+}^{q_{1}},\sigma_{+}^{q_{2}}) +  d_{\mathcal{L}}(\sigma_{-}^{q_{1}},\sigma_{-}^{q_{2}}). \label{1.6}
\end{equation}
Our main theorem is stated as follows:

\begin{theorem}
	\label{t1} Let $\mathcal{Q}$ be equipped with the metric $d$ defined in \eqref{1.6}. For any $q \in \mathcal{Q}$, there exists a unique doubly global classical solution $q(t) \in \mathcal{Q}$ for all $t \in \mathbb{R}$ to the Toda lattice with $q(0) = q$. Moreover, for every $T>0$, if $d(q_{j},q) \to 0$, then
	\[
	\sup_{|t|\leq T} d(q_{j}(t),q(t)) \to 0.
	\]
	The map $t\mapsto q(t)$ is continuous from $\mathbb R$ to
	$(\mathcal Q,d)$.
	Conversely, if a doubly global classical solution with initial datum $q$ exists, then $q \in \mathcal{Q}$.
\end{theorem}

Let $\{ a_{n}(t),b_{n}(t) \}_{n \in \mathbb{Z}}$ be a doubly global classical solution to \eqref{1.1}. Define 
\[
F_{n}(\tau)
=\exp\left( 4\int_0^{\tau}(\tau-s)a_{n}(s)^2 \mathrm{d}s
\right),
\qquad \tau\in\mathbb{R},
\]
which is the key to our proof and satisfies

\begin{theorem}
	\label{p26}There exists a probability measure $\rho_{n}$ on $\mathbb{R}$ satisfying
	\[
	\int_{\mathbb{R}}e^{c|\lambda|}\rho_{n}(\mathrm{d}\lambda)<\infty,
	\qquad \text{for every } c>0,
	\]
	such that
	\[
	F_{n}(\tau)=\int_{\mathbb{R}}e^{2\lambda\tau}\rho_{n}(\mathrm{d}\lambda).
	\]
	Therefore, $F_{n}(z)$ is an entire function and then $4a_{n}(\tau)^2 = (\log F_{n})''(\tau)$ can be extended to a meromorphic function on $\mathbb{C}$ which is holomorphic on a neighborhood of $\mathbb{R}$, and therefore $b_{n}(\tau)$ does as well.
\end{theorem}

Theorem \ref{p26} expresses a basic Toda quantity $F_{n}$ as the bilateral Laplace transform
of a positive measure. Its proof uses the QR representation of the finite Toda
lattice and the exponential of a finite symmetric Jacobi matrix. No limiting
argument for finite-volume solutions is needed: each fixed Jacobi moment and
each time derivative appearing in the argument depends on only finitely many
neighboring coefficients.

Section 2 collects the compound-matrix, finite Toda, and moment-theoretic tools
used below. In Section 3 we prove Theorem \ref{p26}, exponential integrability
of the half-line spectral measures, and uniqueness. Section 4 establishes
compact-time estimates for the boundary Laplace transforms and then proves
existence and continuous dependence by approximation.

We also note the connection with the growth condition used in \cite{AA}.
By the standard moment estimate for Jacobi matrices,
$a_n,|b_n|=o(|n|)$ as $n\to\pm\infty$ implies $q\in\mathcal Q$; see
\cite[Corollary 1]{SK}. Thus Theorem \ref{t1} includes, in particular, the
global-existence result for all sublinear coefficients.

\subsection*{Index convention}
Throughout the paper we use the following convention for the Jacobi
coefficients. The coefficient \(a_n\) is the off-diagonal entry connecting the
sites \(n\) and \(n+1\). Thus the two-sided Jacobi expression associated with
\(q=\{a_n,b_n\}_{n\in\mathbb Z}\) is
\[
(J_q u)_n
=
a_{n}u_{n+1}+a_{n-1}u_{n-1}+b_nu_n,
\qquad n\in\mathbb Z .
\]

For an integer \(k\), cutting the lattice between \(k-1\) and \(k\) gives the
right half-line Jacobi expression \(J^q_{k,+}\) on
\(\ell^2(\{k,k+1,k+2,\ldots\})\), defined by
\[
(J^q_{k,+}u)_n
=
a_{n}u_{n+1}+a_{n-1}u_{n-1}+b_nu_n,
\qquad n\ge k+1,
\]
with boundary equation
\[
(J^q_{k,+}u)_k
=
a_{k}u_{k+1}+b_ku_k .
\]
Similarly, the left half-line Jacobi expression \(J^q_{k,-}\) acts on
\(\ell^2(\{\ldots,k-2,k-1\})\) by
\[
(J^q_{k,-}u)_n
=
a_{n}u_{n+1}+a_{n-1}u_{n-1}+b_nu_n,
\qquad n\le k-2,
\]
with boundary equation
\[
(J^q_{k,-}u)_{k-1}
=
a_{k-2}u_{k-2}+b_{k-1}u_{k-1}.
\]
The coefficient \(a_{k-1}\) is therefore precisely the coupling between the two
half-line operators \(J^q_{k,-}\) and \(J^q_{k,+}\).

Whenever these half-line expressions are essentially self-adjoint, we denote by
\(\sigma^q_{k,+}\) the spectral measure of the self-adjoint closure of
\(J^q_{k,+}\) with respect to the boundary vector \(\delta_k\), and by
\(\sigma^q_{k,-}\) the spectral measure of the self-adjoint closure of
\(J^q_{k,-}\) with respect to the boundary vector \(\delta_{k-1}\). In
particular,
\[
\sigma^q_+ := \sigma^q_{0,+},
\qquad
\sigma^q_- := \sigma^q_{0,-}.
\]

\section{Preliminaries}
\subsection{Compound matrix}
We recall some standard facts about compound matrices; see \cite{RH}. Let $A\in\mathbb{C}^{N\times N}$ and let $1\le d\le N$.  We write
\[
\mathcal{I}_d(N)=\{I=\{i_1<\cdots<i_d\}: I\subset\{1,\ldots,N\}\}.
\]
For $I,J\in\mathcal{I}_d(N)$, let $A[I,J]$ be the submatrix of $A$ obtained by selecting the rows indexed by $I$ and the columns indexed by $J$.

\begin{definition}
	The $d$th multiplicative compound matrix of $A$ is
	\[
	C_d(A)=\bigl(\det A[I,J]\bigr)_{I,J\in\mathcal{I}_d(N)}.
	\]
	The $d$th additive compound matrix is
	\[
	A^{[d]}:= \left. \dfrac{\mathrm{d}}{\mathrm{d}z} C_d(I+zA) \right|_{z=0}.
	\]
\end{definition}

The Cauchy--Binet formula gives
\[
C_d(AB)=C_d(A)C_d(B).
\]
Consequently, $z\mapsto C_d(e^{zA})$ is a one-parameter matrix group whose derivative at $z=0$ is $A^{[d]}$. Hence
\begin{equation}
	C_d(e^{zA})=e^{zA^{[d]}},\qquad z\in\mathbb{C}.
	\label{2.1}
\end{equation}

We shall use the following entry formula.

\begin{lemma}\label{l3}
	Let $A=(a_{ij})_{i,j=1}^N$ and $I=\{i_1<\cdots<i_d\}$, $J=\{j_1<\cdots<j_d\}$.
	Then
	\begin{equation}
		(A^{[d]})_{I,I}=\sum_{r=1}^d a_{i_ri_r}.
		\label{2.2}
	\end{equation}
	If $|I\setminus J|\ge2$, then $(A^{[d]})_{I,J}=0$.  If
	\[
	I\setminus\{i_r\}=J\setminus\{j_s\},
	\]
	then
	\begin{equation}
		(A^{[d]})_{I,J}=(-1)^{r+s}a_{i_rj_s}.
		\label{2.3}
	\end{equation}
	In particular, if $A$ is real symmetric, then $A^{[d]}$ is real symmetric.
\end{lemma}

\begin{proof}
	Differentiate the determinant defining the $(I,J)$ entry of $C_d(I+zA)$ at $z=0$.  If $I=J$, the derivative is the sum of the diagonal entries selected by $I$, giving \eqref{2.2}.  If $I$ and $J$ differ in exactly one position, the only nonzero first-order cofactor is the one obtained by deleting the differing row and column, and its sign is $(-1)^{r+s}$, yielding \eqref{2.3}.  If they differ in two or more positions, every term has order at least two in $z$. \bigskip
\end{proof}

\begin{lemma}\label{l4}
	Let
	\[
	K=D+B\in\mathbb{R}^{N\times N},
	\qquad
	D=\operatorname{diag}(v_1,\ldots,v_N),
	\]
	where
	\begin{equation}
		B=B^T,\qquad B_{ii}=0,\qquad B_{ij}\ge0\quad(i\ne j).
		\label{2.4}
	\end{equation}
	Assume that there exists a diagonal signature matrix
	\[
	\Gamma=\operatorname{diag}(\varepsilon_1,\ldots,\varepsilon_N),
	\qquad \varepsilon_i\in\{-1,1\},
	\]
	for which $\Gamma B\Gamma=-B$. Let $\mathcal S\subset\{1,\ldots,N\}$ and let $o\in\mathcal S$.  Denote by
	$K_{\mathcal S}=K[\mathcal S,\mathcal S]$ the corresponding principal submatrix.  Then
	\begin{equation}
		\bigl(K_{\mathcal S}^{2m}\bigr)_{oo}\le \bigl(K^{2m}\bigr)_{oo},
		\qquad m=0,1,2,\ldots.
		\label{2.5}
	\end{equation}
\end{lemma}

\begin{proof}
	Since $D$ and $B$ need not commute, we use the exact noncommutative expansion
	\begin{equation}
		K^{2m}
		=\sum_{\ell=0}^{2m}
		\sum_{\substack{k_0+\cdots+k_\ell=2m-\ell\\k_j\ge0}}
		D^{k_0}BD^{k_1}B\cdots BD^{k_\ell}.
		\label{2.6}
	\end{equation}
	Every word in $D$ and $B$ of length $2m$ occurs exactly once in \eqref{2.6}.
	
	Because $D$ commutes with $\Gamma$ and $B$ anticommutes with $\Gamma$, one has
	\[
	\Gamma\bigl(D^{k_0}BD^{k_1}\cdots BD^{k_\ell}\bigr)\Gamma
	=(-1)^\ell D^{k_0}BD^{k_1}\cdots BD^{k_\ell}.
	\]
	The $(o,o)$ entry is unchanged under conjugation by $\Gamma$.  Hence every term with odd $\ell$ has zero $(o,o)$ entry.
	
	Fix an even $\ell$ and indices
	\[
	i_0=o,\quad i_1,\ldots,i_{\ell-1},\quad i_\ell=o.
	\]
	For fixed $k_0,\ldots,k_\ell$, the contribution of this index chain equals
	\begin{equation}
		\left(\prod_{r=1}^{\ell}B_{i_{r-1}i_r}\right)
		\left(\prod_{r=0}^{\ell}v_{i_r}^{k_r}\right).
		\label{2.7}
	\end{equation}
	Summing \eqref{2.7} over all $k_r\ge0$ with
	$k_0+\cdots+k_\ell=2m-\ell$ gives
	\begin{equation}
		\left(\prod_{r=1}^{\ell}B_{i_{r-1}i_r}\right)
		h_{2m-\ell}(v_{i_0},\ldots,v_{i_\ell}),
		\label{2.8}
	\end{equation}
	where
	\[
	h_q(x_0,\ldots,x_\ell)
	=\sum_{k_0+\cdots+k_\ell=q}x_0^{k_0}\cdots x_\ell^{k_\ell}
	\]
	is the complete homogeneous polynomial of degree $q$.
	
	For real $x_0,\ldots,x_\ell$ and $r\ge0$,
	\begin{equation}
		h_{2r}(x_0,\ldots,x_\ell)
		=\frac{1}{(2r)!}
		\int_{(0,\infty)^{\ell+1}}
		\left(\sum_{j=0}^{\ell}x_jt_j\right)^{2r}
		e^{-\sum_{j=0}^{\ell}t_j}\,
		\mathrm{d} t_0\cdots \mathrm{d} t_\ell
		\ge0. \label{2.9}
	\end{equation}
	Indeed, expanding the even power and using
	$\int_0^\infty t^ke^{-t} d t=k!$ gives \eqref{2.9}.
	Since $\ell$ and $2m$ are even, $2m-\ell$ is even.  Thus every grouped contribution \eqref{2.8} is nonnegative by \eqref{2.4} and \eqref{2.9}.
	
	The expansion of $(K_{\mathcal S}^{2m})_{oo}$ is identical, except that all intermediate indices $i_1,\ldots,i_{\ell-1}$ are restricted to $\mathcal S$.  Therefore it retains only a subcollection of the nonnegative grouped contributions occurring in $(K^{2m})_{oo}$, which proves \eqref{2.5}. \bigskip
\end{proof}

\subsection{Finite Jacobi matrix}
Fix integers $L<R$.  A finite Toda solution on $[L,R]$ consists of
\[
\widehat a_k(\tau),\quad L\le k\le R-1,
\qquad
\widehat b_k(\tau),\quad L\le k\le R,
\]
satisfying
\begin{equation}
	\left\{
	\begin{array}
		[c]{l}%
		\dot{\widehat a}_k=\widehat a_k(\widehat b_{k+1}-\widehat b_k),
		\qquad L\le k\le R-1, \\
		\dot{\widehat b}_k=2(\widehat a_k^2-\widehat a_{k-1}^2),
		\qquad L\le k\le R,
	\end{array}
	\right. \label{2.10}
\end{equation}
with open-boundary convention
\begin{equation}
	\widehat a_{L-1}=\widehat a_R=0.
	\label{2.11}
\end{equation}
Its Jacobi matrix is
\[
\widehat J(\tau)=
\begin{pmatrix}
	\widehat b_L&\widehat a_L&&0\\
	\widehat a_L&\widehat b_{L+1}&\cdots&\\
	&\cdots&\cdots&\widehat a_{R-1}\\
	0&&\widehat a_{R-1}&\widehat b_R
\end{pmatrix}.
\]

We use the standard QR representation of the finite Toda flow \cite{WS}. The integrability of the Toda lattice gives that
\begin{equation}
	\widehat J(\tau)=Q(\tau)^T J_0 Q(\tau), \quad \text{with } J_{0} = \widehat{J}(0),
	\label{2.12}
\end{equation}
with a unitary matrix $Q(\tau)$ satisfying
\[
\dot{Q}(\tau) = -Q(\tau)B(\tau), \quad Q(0) = I
\]
where $B(\tau)$  is the skew-symmetrization of $\widehat{J}(\tau)$. Let $R(\tau)$ be the solution to
\[
\dot{R}(\tau) = \left( \widehat{J}(\tau)+B(\tau) \right) R(\tau), \quad R(0) = I.
\]
Then $R(\tau)$ is a upper triangular matrix with positive diagonal. Moreover, $Q(\tau)R(\tau)$ satisfies
\[
\dfrac{\mathrm{d}}{\mathrm{d}\tau} Q(\tau)R(\tau) = \dot{Q}(\tau)R(\tau) + Q(\tau)\dot{R}(\tau) = Q(\tau) \widehat{J}(\tau) R(\tau) = J_{0} Q(\tau)R(\tau),
\]
which implies $e^{\tau J_0}=Q(\tau)R(\tau)$. Differentiating $e^{\tau J_0}=Q(\tau)R(\tau)$ gives
$R'R^{-1}=\widehat J-Q^TQ'$.  The left-hand side is upper triangular, $\widehat J$ is symmetric, and $Q^TQ'$ is skew-symmetric, so its diagonal equals that of $\widehat J$, which yields
\begin{equation}
	\frac{\mathrm{d}}{\mathrm{d}\tau}\log r_k(\tau)=\widehat b_k(\tau).
	\label{2.13}
\end{equation}
where $r_k(\tau)$ denotes the diagonal entry of $R(\tau)$ corresponding to the site $k$.

Fix $n$ with $L\le n<R$ and put
\[
I_0=\{L,L+1,\ldots,n\},
\qquad d=n-L+1.
\]
The set $I_0$ indexes the leading $d\times d$ principal block. Note that
\[
e^{2\tau J_{0}} = \left( e^{\tau J_{0}} \right)^{T} e^{\tau J_{0}} = R(\tau)^{T} Q(\tau)^{T} Q(\tau) R(\tau) = R(\tau)^{T} R(\tau).
\]
Since $R(\tau)$ is upper triangular,
\begin{equation}
	\det\bigl(e^{2\tau J_0}[I_0,I_0]\bigr)
	=\prod_{k=L}^{n}r_k(\tau)^2.
	\label{2.14}
\end{equation}
Thus
\begin{equation}
	\frac{\mathrm{d}}{\mathrm{d}\tau}
	\log\det\bigl(e^{2\tau J_0}[I_0,I_0]\bigr)
	=2\sum_{k=L}^{n}\widehat b_k(\tau).
	\label{2.15}
\end{equation}

\subsection{The Laplace transform}
Recall that the bilateral Laplace transform is defined by
\[
\mathcal{L}_\mu(z)=\int_{\mathbb{R}}e^{xz}\mu(\mathrm{d} x),\qquad z\in\mathbb{C}
\]
for $\mu\in\mathcal P_{\exp}(\mathbb{R})$. A metric on $\mathcal P_{\exp}(\mathbb{R})$ is given by
\[
d_{\mathcal{L}}(\mu,\nu)
:=\sum_{m=1}^{\infty}2^{-m}
\vartheta\!\left(p_m(\mathcal{L}_\mu-\mathcal{L}_\nu)\right)
\]
with $p_R(f):=\sup_{|z|\le R}|f(z)|$ and $\vartheta(s)=\frac{s}{1+s}$.

\begin{lemma}\label{l6}
	Let $\{\mu_j\}_{j\geq1}$ be a sequence of probability measures on
	$\mathbb R$ such that, for some constants $c,C>0$,
	\begin{equation}
		\int_{\mathbb R}e^{c|\lambda|}\mu_j(\mathrm d\lambda)\leq C,
		\qquad j\geq1. \label{2.16}
	\end{equation}
	Suppose that there exists a probability measure $\mu$ on $\mathbb R$ such
	that, for every integer $k\geq0$,
	\begin{equation}
		\int_{\mathbb R}\lambda^k\mu_j(\mathrm d\lambda)
		\longrightarrow
		\int_{\mathbb R}\lambda^k\mu(\mathrm d\lambda). \label{2.17}
	\end{equation}
	Then
	\begin{equation}
		\int_{\mathbb R}e^{c|\lambda|}\mu(\mathrm d\lambda)\leq C,
		\qquad
		\sup_{|z|\leq c'}
		|\mathcal L_{\mu_j}(z)-\mathcal L_\mu(z)|\longrightarrow0
		\quad (0<c'<c). \label{2.18}
	\end{equation}
\end{lemma}

\begin{proof}
	The functions $\mathcal L_{\mu_j}$ are analytic in the strip
	$S_c=\{z\in\mathbb C:|\operatorname{Re}z|<c\}$ and are locally uniformly
	bounded there by \eqref{2.16}. Hence every subsequence contains a further
	subsequence, denoted again by $\mathcal L_{\mu_j}$, which converges locally
	uniformly on $S_c$ to an analytic function $f$. By Cauchy's formula and
	\eqref{2.17},
	\[
	f^{(k)}(0)=\int_{\mathbb R}\lambda^k\mu(\mathrm d\lambda),
	\qquad k\geq0.
	\]
	For $0<c''<c$, Fatou's lemma applied to the even moment series gives
	\[
	\begin{aligned}
		\int_{\mathbb R}\cosh(c''\lambda)\mu(\mathrm d\lambda)
		&=\sum_{m=0}^{\infty}\frac{(c'')^{2m}}{(2m)!}
		\int_{\mathbb R}\lambda^{2m}\mu(\mathrm d\lambda)\\
		&\leq\liminf_{j\to\infty}
		\int_{\mathbb R}\cosh(c''\lambda)\mu_j(\mathrm d\lambda)<\infty.
	\end{aligned}
	\]
	Thus $\mathcal L_\mu$ is analytic on $S_c$, and the equality of all
	derivatives at zero implies $f=\mathcal L_\mu$ there. Since every convergent
	subsequence has the same limit, the whole sequence converges locally
	uniformly on $S_c$, which proves the second assertion in \eqref{2.18}.
	
	In particular, the characteristic functions converge pointwise, so
	$\mu_j$ converges weakly to $\mu$. The first assertion in \eqref{2.18}
	now follows from \eqref{2.16} and the Portmanteau theorem. \bigskip
\end{proof}

We shall also use the standard correspondence between probability measures
with infinite support, orthogonal polynomials, and half-line Jacobi matrices;
see, for example, \cite[Chapters I--II]{TC}.

\begin{lemma}\label{l7}
	The convergence $d(q_j,q)\to0$ holds for $q_j,q\in\mathcal Q$ if and only if
	\begin{equation}
		\sup_j\int_{\mathbb R}e^{c|\lambda|}
		\sigma_\pm^{q_j}(\mathrm d\lambda)<\infty
		\qquad\text{for every }c>0, \label{3.23}
	\end{equation}
	and
	\begin{equation}
		a_n(q_j)\longrightarrow a_n(q),\qquad
		b_n(q_j)\longrightarrow b_n(q)
		\qquad\text{for every }n\in\mathbb Z. \label{3.24}
	\end{equation}
\end{lemma}

\begin{proof}
	Suppose first that $d(q_j,q)\to0$. Then the boundary Laplace transforms
	converge locally uniformly. Evaluating them at $\pm c$ gives \eqref{3.23},
	while Cauchy's formula gives convergence of every moment. The Jacobi
	coefficients are continuous functions of the Hankel determinants of these
	moments, and therefore \eqref{3.24} follows on each half-line; the coupling
	coefficient $a_{-1}$ converges by the definition of $d$.
	
	Conversely, suppose that \eqref{3.23}--\eqref{3.24} hold. Every boundary
	moment is a polynomial in finitely many Jacobi coefficients, so
	\eqref{3.24} implies convergence of all moments of
	$\sigma_\pm^{q_j}$ to those of $\sigma_\pm^q$. For any $R>0$, apply Lemma
	\ref{l6} with some $c>R$ to obtain
	\[
	p_R\left(\mathcal L_{\sigma_\pm^{q_j}}
	-\mathcal L_{\sigma_\pm^q}\right)\longrightarrow0.
	\]
	Together with $a_{-1}(q_j)\to a_{-1}(q)$, this is exactly
	$d(q_j,q)\to0$. \bigskip
\end{proof}

\section{Exponential integrability of spectral measures}
In this section, our primary goal is to prove that if a doubly global classical solution $q(t) = \{ a_{n}(t),b_{n}(t) \}_{n \in \mathbb{Z}}$ to the Toda lattice exists, then $q(t)$ must belong to $\mathcal{Q}$ for every time $t$, and the solution is unique if it exists.

Fix a global time $t_0\in\mathbb{R}$, a cut $n\in\mathbb{Z}$, and finite endpoints $L<n<R$. Let
\[
J_0=J_{L,R}(t_0)
\]
be the finite Jacobi matrix on $[L,R]$ with
\[
\widehat b_k(0)=b_k(t_0),\qquad L\leq k\leq R,
\quad\text{and}\quad
\widehat a_k(0)=a_{k}(t_0),\qquad L\leq k<R.
\]
Let $\widehat J(\tau)$ be the finite Toda solution generated by $J_0$.
Define
\[
\Theta_{L,R,n,t_0}(\tau)
=\exp\left(-2\tau\sum_{k=L}^{n-1}b_k(t_0)\right)
\det\bigl(e^{2\tau J_0}[I_0,I_0]\bigr).
\]
It is strictly positive and satisfies
\begin{equation}
	\Theta_{L,R,n,t_0}(0)=1,
	\qquad
	(\log\Theta_{L,R,n,t_0})'(0)=2b_n(t_0),
	\label{3.2}
\end{equation}
while \eqref{2.15} and \eqref{2.11} give
\begin{equation}
	(\log\Theta_{L,R,n,t_0})''(\tau)
	=2\sum_{k=L}^{n}\dot{\widehat b}_k(\tau)
	=4\sum_{k=L}^{n}\bigl(\widehat a_k(\tau)^2-\widehat a_{k-1}(\tau)^2\bigr)
	=4\widehat a_n(\tau)^2.
	\label{3.3}
\end{equation}

The compound-matrix representation of $\Theta$ is crucial. By
\eqref{2.1},
\[
\det\bigl(e^{2\tau J_0}[I_0,I_0]\bigr)
=\bigl(C_d(e^{2\tau J_0})\bigr)_{I_0,I_0}
=\bigl(e^{2\tau J_0^{[d]}}\bigr)_{I_0,I_0}.
\]
Let $e_{I_0}$ denote the coordinate vector corresponding to $I_0$ and set
\begin{equation}
	K_{L,R,n,t_0}
	=2J_0^{[d]}-2\left(\sum_{k=L}^{n-1}b_k(t_0)\right)I.
	\label{3.4}
\end{equation}
Then
\[
\Theta_{L,R,n,t_0}(\tau) = \left\langle e_{I_0},e^{\tau K_{L,R,n,t_0}}e_{I_0} \right\rangle.
\]
In particular,
\[
\Theta_{L,R,n,t_0}^{(r)}(0)=\left\langle e_{I_0},K_{L,R,n,t_0}^r e_{I_0} \right\rangle.
\]
Since $K_{L,R,n,t_0}$ is real symmetric, for every $M\ge0$ the Hankel matrix
\[
\left[\Theta_{L,R,n,t_0}^{(p+q)}(0)\right]_{p,q=0}^M
\]
is positive semidefinite. Indeed, for $c_0,\cdots,c_M\in\mathbb{C}$,
\[
\sum_{p,q=0}^{M}\overline{c_p}c_q
\Theta_{L,R,n,t_0}^{(p+q)}(0)
=\left\Vert \sum_{p=0}^{M}c_pK_{L,R,n,t_0}^pe_{I_0} \right\Vert^{2} \ge0.
\]

We now pass from finite Taylor data to the global lattice.  Because the right-hand sides of Toda lattice are local polynomials, every coordinate of a classical solution is $C^\infty$ in time, and each fixed-order time derivative at a site depends only on finitely many coefficients in a neighborhood of that site.

\begin{lemma}\label{l10}
	For every $r\ge0$ there exists an integer $N_r$ such that the values
	$a_n^{(r)}(t_0)$ and $b_n^{(r)}(t_0)$ are universal polynomials in
	\[
	\{a_j(t_0),b_j(t_0): |j-n|\le N_r\}.
	\]
	Consequently, for each $M$, if $L$ and $R$ are sufficiently far from $n$, then the finite solution generated by $J_{L,R}(t_0)$ and the global solution have identical time derivatives at the relevant sites through every order needed to compute the first $M$ Taylor coefficients below.
\end{lemma}

Define the global auxiliary function
\[
F_{n,t_0}(\tau)
=\exp\left(
2\tau b_n(t_0)
+4\int_0^{\tau}(\tau-s)a_{n}(t_0+s)^2 \mathrm{d}s
\right),
\qquad \tau\in\mathbb{R}.
\]
Thus
\[
F_{n,t_0}(0)=1,
\qquad
(\log F_{n,t_0})'(0)=2b_n(t_0),
\qquad
(\log F_{n,t_0})''(\tau)=4a_{n}(t_0+\tau)^2.
\]
Combining \eqref{3.2}, \eqref{3.3}, and Lemma \ref{l10}, we obtain: for every $M$ and all sufficiently distant $L,R$,
\begin{equation}
	\Theta_{L,R,n,t_0}^{(r)}(0)=F_{n,t_0}^{(r)}(0),
	\qquad 0\le r\le 2M.
	\label{3.5}
\end{equation}
Therefore
\begin{equation}
	\left[F_{n,t_0}^{(p+q)}(0)\right]_{p,q=0}^{M}\ge0,
	\qquad M=0,1,2,\ldots.
	\label{3.6}
\end{equation}

The same positivity holds at every real point. Indeed, for fixed $x\in\mathbb{R}$, the two functions
$h\mapsto F_{n,t_0}(x+h)$ and $h\mapsto F_{n,t_0 +x}(h)$ have the same second logarithmic derivative.  Hence there exists $\gamma\in\mathbb{R}$ such that
\begin{equation}
	F_{n,t_0}(x+h)=F_{n,t_0}(x)e^{\gamma h}F_{n,t_0+x}(h).
	\label{3.7}
\end{equation}
If $G(h)=e^{\gamma h}H(h)$ and
\[
T_{pr}=\begin{cases}
	\binom{p}{r}\gamma^{p-r},&r\le p,\\
	0,&r>p,
\end{cases}
\]
then the Leibniz and Vandermonde identities give
\begin{equation}
	[G^{(p+q)}(0)]_{p,q=0}^{M}
	=T[H^{(p+q)}(0)]_{p,q=0}^{M}T^T.
	\label{3.8}
\end{equation}
Using \eqref{3.7}, \eqref{3.6}, and
\eqref{3.8}, we conclude that
\begin{equation}
	\left[F_{n,t_0}^{(p+q)}(x)\right]_{p,q=0}^{M}\ge0
	\quad\text{for every }x\in\mathbb{R}\text{ and }M\ge0.
	\label{3.9}
\end{equation}

Generally we record a self-contained version of the exponentially convex function theorem.

\begin{lemma}\label{l11}
	Let $F\in C^\infty(\mathbb{R})$ satisfy $F(0)=1$ and
	\begin{equation}
		[F^{(p+q)}(x)]_{p,q=0}^{M}\ge0
		\qquad (x\in\mathbb{R},\ M\ge0).
		\label{3.10}
	\end{equation}
	Then there exists a probability measure $\rho$ on $\mathbb{R}$ such that
	\begin{equation}
		F(x)=\int_{\mathbb{R}}e^{2x\lambda}\rho(\mathrm{d}\lambda),
		\qquad x\in\mathbb{R},
		\label{3.11}
	\end{equation}
	and
	\begin{equation}
		\int_{\mathbb{R}}e^{c|\lambda|}\rho(\mathrm{d}\lambda)<\infty
		\qquad\text{for every }c>0.
		\label{3.12}
	\end{equation}
\end{lemma}

\begin{proof}
	From the diagonal entries of \eqref{3.10},
	\begin{equation}
		F^{(2k)}(x)\ge0.
		\label{3.13}
	\end{equation}
	Fix $h>0$.  Taylor's theorem with integral remainder, applied to $F(x+h)$ and $F(x-h)$ through order $2N+1$, and \eqref{3.13} yield
	\begin{equation}
		F(x+h)+F(x-h)
		\ge2\sum_{k=0}^{N}\frac{F^{(2k)}(x)}{(2k)!}h^{2k}.
		\label{3.22}
	\end{equation}
	Consequently,
	\begin{equation}
		F^{(2k)}(x)
		\le \frac{(2k)!}{2h^{2k}}\bigl(F(x+h)+F(x-h)\bigr).
		\label{3.14}
	\end{equation}
	The principal $2\times2$ minor indexed by $k,k+1$ gives
	\begin{equation}
		|F^{(2k+1)}(x)|^2
		\le F^{(2k)}(x)F^{(2k+2)}(x).
		\label{3.15}
	\end{equation}
	To make the analyticity conclusion explicit, fix a compact interval $I$ and choose $h>0$.  On the enlarged compact interval $I_h=\{x:\operatorname{dist}(x,I)\le h\}$, the right-hand side of \eqref{3.14} is uniformly bounded.  Equations \eqref{3.14}--\eqref{3.15} therefore imply
	\[
	\sup_{x\in I}|F^{(j)}(x)|\le C_{I,h}\,j!\,h^{-j},
	\qquad j\ge0,
	\]
	after enlarging $C_{I,h}$ by an absolute factor.  Applying Taylor's theorem on a smaller interval, say $|y-x|<h/2$, makes the remainder bounded by $C_{I,h}2^{-N}$ at order $N$.  Thus $F$ equals its Taylor series locally and is real analytic on $\mathbb{R}$.
	
	Set
	\[
	s_k=2^{-k}F^{(k)}(0).
	\]
	All Hankel matrices $[s_{p+q}]_{p,q=0}^{M}$ are positive semidefinite. By the Hamburger moment theorem, there exists a probability measure $\rho$ such that
	\[
	s_k=\int_{\mathbb{R}}\lambda^k\rho(\mathrm{d}\lambda),
	\qquad k\ge0.
	\]
	Taking $x=0$ and $h=c/2$ in \eqref{3.22}, letting $N\to\infty$, and using monotone convergence gives
	\[
	\int_{\mathbb{R}}\cosh(c\lambda)\rho(\mathrm{d}\lambda)
	=\sum_{k=0}^{\infty}\frac{c^{2k}s_{2k}}{(2k)!}
	\le \frac{F(c/2)+F(-c/2)}{2}<\infty.
	\]
	Since $e^{c|\lambda|}\le2\cosh(c\lambda)$, \eqref{3.12} follows.
	
	The function
	\[
	G(z)=\int_{\mathbb{R}}e^{2\lambda z}\rho(\mathrm{d}\lambda)
	\]
	is entire, by \eqref{3.12}, and satisfies $G^{(k)}(0)=F^{(k)}(0)$ for every $k$.  Since $F$ is real analytic, $F=G$ near zero and hence on all of $\mathbb{R}$.  This proves \eqref{3.11}. \bigskip
\end{proof}

\begin{proof}[Proof of Theorem \ref{p26}]
	Applying Lemma \ref{l11}, there exists a probability measure $\rho_{n,t_{0}}$ satisfying
	\[
	\int_{\mathbb{R}}e^{c|\lambda|}\rho_{n,t_0}(\mathrm{d}\lambda)<\infty,
	\qquad \text{for every } c>0,
	\]
	such that
	\[
	F_{n,t_{0}}(\tau)=\int_{\mathbb{R}}e^{2\lambda\tau}\rho_{n,t_{0}}(\mathrm{d}\lambda).
	\]
	Since $F_n(\tau)=e^{-2\tau b_n(0)}F_{n,0}(\tau)$, translating
	$\rho_{n,0}$ by $-b_n(0)$ gives the measure $\rho_n$ in the statement.
	Thus $F_n$ is entire. Its zeros do not meet the real axis, because
	$F_n(\tau)>0$ for $\tau\in\mathbb R$, and hence
	$4a_n^2=(\log F_n)''$ is meromorphic on $\mathbb C$ and holomorphic on a
	neighborhood of $\mathbb R$. Moreover, the second Toda equation gives
	\[
	b_n(t)=b_n(0)+\frac12(\log F_n)'(t)
	-\frac12(\log F_{n-1})'(t),
	\]
	so $b_n$ has the same meromorphic extension property. \bigskip
\end{proof}

Regarding the relationship between the limit-point property of Jacobi operators or their self-adjointness and $q \in \mathcal{Q}$, we have the following lemma:

\begin{lemma}\label{l12}
	Let $J_+$ be a Jacobi expression on $\ell^2(\mathbb Z_+)$ and let
	\[
	s_k=\left\langle J_+^k\delta_0,\delta_0\right\rangle,
	\qquad k\geq0.
	\]
	Assume that, for some constants $A,c>0$,
	\begin{equation}
		s_{2k}\leq A\frac{(2k)!}{c^{2k}},
		\qquad k\geq0. \label{3.25}
	\end{equation}
	Then $J_+$ is essentially self-adjoint, and its spectral measure is the
	unique probability measure $\sigma_+$ with moments $s_k$. Moreover,
	\[
	\int_{\mathbb R}e^{c'|\lambda|}\sigma_+(\mathrm d\lambda)<\infty
	\qquad\text{for every }0<c'<c.
	\]
\end{lemma}

\begin{proof}
	The Hankel matrices of $\{s_k\}_{k\geq0}$ are positive semidefinite, since
	\[
	\sum_{p,q=0}^M\overline{\alpha_p}\alpha_qs_{p+q}
	=\left\|\sum_{p=0}^M\alpha_pJ_+^p\delta_0\right\|^2.
	\]
	Hence the Hamburger moment theorem gives a probability measure with these
	moments. Estimate \eqref{3.25} implies Carleman's condition, so the moment
	problem is determinate and the minimal Jacobi operator is essentially
	self-adjoint; see \cite{NA,BS}. Finally,
	\[
	\int_{\mathbb R}\cosh(c'\lambda)\sigma_+(\mathrm d\lambda)
	=\sum_{k=0}^{\infty}\frac{(c')^{2k}s_{2k}}{(2k)!}
	\leq A\sum_{k=0}^{\infty}\left(\frac{c'}c\right)^{2k}<\infty,
	\]
	which proves the last assertion. \bigskip
\end{proof}

We write
\begin{equation}
	r_{2m}(n,t_0)
	:=\int_{\mathbb{R}}\lambda^{2m}\rho_{n,t_0}(\mathrm{d}\lambda)
	=2^{-2m}F_{n,t_0}^{(2m)}(0).
	\label{3.16}
\end{equation}

Let $\mathcal{I}_{d}([L,R])$ denote the $d$-element subsets of $\{L,L+1,\cdots,R\}$.  For
$I=\{i_1<\cdots<i_d\}$ define
\[
\Gamma_{I,I}=(-1)^{i_1+\cdots+i_d}.
\]
Write
\[
K_{L,R,n,t_0}=D+B,
\]
where $D$ is diagonal and $B$ has zero diagonal.  By Lemma \ref{l3}, a nonzero off-diagonal entry of $J_0^{[d]}$ can occur only when one selected index is replaced by an adjacent unselected index.  In that case the replaced index has the same position in the two ordered sets, so the sign in \eqref{2.3} is positive.  Hence
\[
B=B^T,\qquad B_{I,J}\ge0.
\]
Moreover, the sum of the selected indices changes by one, so
\[
\Gamma B\Gamma=-B.
\]
Thus Lemma \ref{l4} applies.

For $n\le k\le R$, define
\[
I_k=\{L,L+1,\ldots,n-1,k\}
\]
and let
\[
\mathcal S=\{I_n,I_{n+1},\ldots,I_R\}.
\]
Using Lemma \ref{l3} and \eqref{3.4}, we obtain
\[
(K_{L,R,n,t_0})_{I_k,I_k}=2b_k(t_0),
\]
and
\[(K_{L,R,n,t_0})_{I_k,I_{k+1}}=2a_{k}(t_0),
\qquad n\le k<R,
\]
and all remaining off-diagonal entries inside $\mathcal S$ vanish.  Therefore
\begin{equation}
	K_{L,R,n,t_0}[\mathcal S,\mathcal S]
	=2J_{[n,R]}(t_0),
	\label{3.17}
\end{equation}
where $J_{[n,R]}(t_0)$ is the finite Jacobi matrix on $[n,R]$.

Since $I_n=I_0$, Lemma \ref{l4}, \eqref{3.5} and \eqref{3.17} imply
\begin{equation}
	2^{2m}\left\langle \delta_n,J_{[n,R]}(t_0)^{2m}\delta_n \right\rangle
	\le
	\left\langle e_{I_0},K_{L,R,n,t_0}^{2m}e_{I_0} \right\rangle
	=\Theta_{L,R,n,t_0}^{(2m)}(0) = F_{n,t_{0}}^{(2m)}(0).
	\label{3.18}
\end{equation}
\begin{lemma}\label{l19}
	For every doubly global classical solution and every $n\in\mathbb Z$,
	$t_0\in\mathbb R$, the right boundary spectral measure satisfies
	\begin{equation}
		\begin{aligned}
			\int_{\mathbb R}\cosh(c\lambda)
			\sigma_{n,+}^{q(t_0)}(\mathrm d\lambda)
			&\leq\int_{\mathbb R}\cosh(c\lambda)
			\rho_{n,t_0}(\mathrm d\lambda)\\
			&=\frac{F_{n,t_0}(c/2)+F_{n,t_0}(-c/2)}2,
			\qquad c>0.
		\end{aligned} \label{3.26}
	\end{equation}
	The analogous assertion holds for the left boundary spectral measure.
\end{lemma}

\begin{proof}
	Let $J_{n,+}(t_0)$ be the Jacobi expression on the right half-line, acting on finitely supported sequences.  Define its formal moments by
	\begin{equation}
		s_{k,+}^{(n,t_0)}
		:= \left\langle \delta_n,(J_{n,+}(t_0))^k\delta_n \right\rangle.
		\label{3.19}
	\end{equation}
	For fixed $m$, the vector $(J_{n,+}(t_0))^m\delta_n$ has finite support, so the left-hand side of \eqref{3.18} stabilizes at
	$2^{2m}s_{2m,+}^{(n,t_0)}$ as $R\to\infty$.  Taking $L,R$ sufficiently far so that \eqref{3.5} holds through order $2m$, we obtain from \eqref{3.16} and \eqref{3.18}
	\begin{equation}
		s_{2m,+}^{(n,t_0)}\le r_{2m}(n,t_0)
		\qquad (m\ge0).
		\label{3.20}
	\end{equation}
	
	For every $c_0>0$, choose $c>c_0$ and put
	$A=\int_{\mathbb R}e^{c|\lambda|}\rho_{n,t_0}(\mathrm d\lambda)$.
	Then
	\[
	r_{2m}(n,t_0)\leq A\frac{(2m)!}{c^{2m}}.
	\]
	Lemma \ref{l12} and \eqref{3.20} show that $J_{n,+}(t_0)$ is essentially
	self-adjoint and that its spectral measure has the formal moments
	\eqref{3.19}. Therefore monotone convergence and \eqref{3.20} give
	\begin{align*}
		\int_{\mathbb{R}}\cosh(c_0\lambda)
		\sigma_{n,+}^{q(t_0)}(\mathrm{d}\lambda)
		&=\sum_{m=0}^{\infty}\frac{c_0^{2m}}{(2m)!}
		s_{2m,+}^{(n,t_0)}\\
		&\le\sum_{m=0}^{\infty}\frac{c_0^{2m}}{(2m)!}r_{2m}(n,t_0)\\
		&=\int_{\mathbb{R}}\cosh(c_0\lambda)
		\rho_{n,t_0}(\mathrm{d}\lambda).
	\end{align*}
	This is \eqref{3.26}. The left-half-line assertion follows by the same
	argument. \bigskip
\end{proof}

\begin{proof}[Proof of exponential integrability]
	By Lemma \ref{l19} and $e^{c|\lambda|}\leq2\cosh(c\lambda)$, both boundary
	spectral measures of $q(t_0)$ have exponential moments of every order.
	Hence $q(t_0)\in\mathcal Q$. Since $t_0$ was arbitrary,
	$q(t)\in\mathcal Q$ for every $t\in\mathbb R$. \bigskip
\end{proof}

\begin{proof}[Proof of Uniqueness]
	Let $q^{(1)}(t)$ and $q^{(2)}(t)$ be two
	global solutions with the same initial datum, and let $F_{n,0}^{(j)}$ be the
	corresponding auxiliary functions. By Lemma \ref{l10}, all
	derivatives of $F_{n,0}^{(1)}$ and $F_{n,0}^{(2)}$ at the origin coincide.
	Their bilateral Laplace representations extend both functions to entire
	functions, and therefore
	\[
	F_{n,0}^{(1)}(z)=F_{n,0}^{(2)}(z),\qquad z\in\mathbb{C}.
	\]
	Taking the second logarithmic derivative and using positivity of the
	off-diagonal coefficients gives
	\[
	a_{n}^{(1)}(t)=a_{n}^{(2)}(t),
	\qquad n\in\mathbb{Z},\quad t\in\mathbb{R}.
	\]
	Finally, the second Toda equation and the common initial datums give
	\[
	b_n^{(j)}(t)=b_n(0)+2\int_0^t
	\left((a_{n}^{(j)}(s))^2-(a_{n-1}^{(j)}(s))^2\right)\mathrm{d}s,
	\]
	so $b_n^{(1)}(t)=b_n^{(2)}(t)$.
\end{proof}

\section{Global existence and well-posedness}
In this section, we will prove the global existence of doubly classical solution to the Toda lattice with initial datum $q \in \mathcal{Q}$, as well as the continuous dependence of the solutions on the initial data, thereby completing the proof of the main theorem.

For a doubly global classical solution $q(t)$, let
\begin{equation}
	L_+(z,t)=\mathcal L_{\sigma_+^{q(t)}}(z),
	\qquad
	L_-(z,t)=\mathcal L_{\sigma_-^{q(t)}}(z). \label{4.1}
\end{equation}
For entire functions $f$ and $g$, set
\[
\mathcal C(f,g)(z)
=
\int_0^z f(\zeta)g(z-\zeta)\,\mathrm d\zeta
=
z\int_0^1 f(\theta z)g((1-\theta)z)\,\mathrm d\theta .
\]

The boundary Laplace transforms defined in \eqref{4.1} satisfy the following equations:

\begin{lemma}\label{l14}
	Along every doubly global classical Toda trajectory, the boundary Laplace
	transforms satisfy
	\begin{equation}
		\begin{split}
			\partial_tL_+
			&=2\partial_zL_+-2b_0L_+
			-2a_{-1}^2\mathcal C(L_+,L_+),\\
			\partial_tL_-
			&=-2\partial_zL_-+2b_{-1}L_-
			+2a_{-1}^2\mathcal C(L_-,L_-).
		\end{split} \label{4.16}
	\end{equation}
	The equalities hold locally uniformly in $z$, uniformly for $t$ in compact
	intervals.
\end{lemma}

\begin{proof}
	Let $J_+(t)$ and $J_-(t)$ be the two half-line Jacobi matrices.  Let $P(t)$ be
	the skew-symmetrization of $J(t)$
	\[
	P_{n,n+1}=a_{n},\qquad P_{n+1,n}=-a_{n},
	\]
	and let $P_\pm$ be its half-line restrictions.  If
	$E_+=\left| \delta_{0} \right\rangle \left\langle \delta_{0} \right|$ and
	$E_-=\left| \delta_{-1} \right\rangle \left\langle \delta_{-1} \right|$, direct multiplication at the
	boundary gives the Lax equations
	\begin{equation}
		\dot J_+=[P_+,J_+]-2a_{-1}^2E_+,
		\qquad
		\dot J_-=[P_-,J_-]+2a_{-1}^2E_-.
		\label{4.14}
	\end{equation}
	For
	\[
	s_{k,+}=\left\langle \delta_0,J_+^k\delta_0 \right\rangle,
	\qquad
	s_{k,-}=\left\langle \delta_{-1},J_-^k\delta_{-1} \right\rangle,
	\]
	the product rule and telescoping of the commutator give
	\begin{equation}
		\left\{
		\begin{array}
			[l]{l}%
			\dot s_{k,+}
			=2(s_{k+1,+}-b_0 s_{k,+})
			-2a_{-1}^2\sum_{\ell=0}^{k-1}s_{\ell,+}s_{k-1-\ell,+} \\
			\dot s_{k,-}
			=-2(s_{k+1,-}-b_{-1}s_{k,-})
			+2a_{-1}^2\sum_{\ell=0}^{k-1}s_{\ell,-}s_{k-1-\ell,-}
		\end{array}
		\right. . \label{4.15}
	\end{equation}
	where the sum is empty when $k=0$.  Indeed, in the plus case the commutator
	part contributes
	\[
	2\bigl(s_{k+1,+}-b_0 s_{k,+}\bigr),
	\]
	and the missing coupling across the cut contributes the convolution sum with
	coefficient $-2a_{-1}^2$; the signs are reversed on the left half-line.
	
	Note that if
	$f(z)=\sum f_kz^k/k!$ and $g(z)=\sum g_kz^k/k!$, then
	\[
	\mathcal{C}(f,g)(z)
	=\sum_{k=1}^{\infty}
	\left(\sum_{\ell=0}^{k-1}f_\ell g_{k-1-\ell}\right)\frac{z^k}{k!}.
	\]
	and we have
	\[
	L_{\pm}(z,t) = \sum_{k=0}^{\infty} s_{k,\pm}(t) \dfrac{z^{k}}{k!}.
	\]
	Then multiplying \eqref{4.15} by $z^k/k!$
	and summing gives \eqref{4.16}.  The locally uniform exponential bounds
	obtained from Lemma \ref{l19}
	justify all differentiations and summations uniformly on compact time
	intervals and compact $z$-sets. \bigskip
\end{proof}

The following estimate is the basic compact-time bound used below.  Its
form is exactly adapted to the fixed cut between $-1$ and $0$.

\begin{lemma}\label{l20}
	Let $q(t)$ be a doubly global classical Toda solution, with coefficients
	$\{a_n(t),b_n(t)\}_{n\in\mathbb Z}$, and let $L_\pm(z,t)$ be its two boundary
	Laplace transforms.
	For every $T>0$ and $|t|\le T$,
	\begin{equation}
		a_{-1}(t)^2
		\le a_{-1}(0)^2L_+(2t,0)L_-(-2t,0).\label{4.2}
	\end{equation}
	Moreover, for $t\in[0,T]$ and $t\in[-T,0]$, respectively,
	\begin{equation}
		\begin{gathered}
			|b_0(t)|,\ |b_{-1}(t)|
			\le \frac1T\log\bigl(2N_\pm(T)\bigr),\\
			\max\left\{
			\sup_{|x|\le3T}L_+(x,t),
			\sup_{|x|\le3T}L_-(x,t)
			\right\}
			\le N_\pm(T),
		\end{gathered} \label{4.3}
	\end{equation}
	where $c_0=10e^2$,
	\[
	\mathfrak M(T)
	=\max_{\varepsilon\in\{+,-\}}
	\int_{\mathbb R}e^{5T|\lambda|}
	\sigma_\varepsilon^{q(0)}(\mathrm d\lambda),
	\]
	and
	\[
	\log N_\pm(T)
	=c_0\log\mathfrak M(T)+c_0
	+c_0T\left|\int_0^{\pm T}a_{-1}(s)^2\,\mathrm ds\right|.
	\]
\end{lemma}

\begin{proof}
	\medskip
	\noindent
	\textit{1. The mixed characteristic estimate.}
	We first treat $0\le t\le T$.  Fix such a $t$ and, for $0\le s\le t$, set
	\[
	x_+(s)=2(t-s),\qquad x_-(s)=-2(t-s),
	\]
	\[
	S(s)=a_{-1}(s)^2L_+(x_+(s),s)L_-(x_-(s),s).
	\]
	Along these two characteristics, Lemma \ref{l14} gives
	\[
	\frac{\mathrm d}{\mathrm ds}\log L_+(x_+(s),s)
	=-2b_0(s)-2a_{-1}(s)^2
	\frac{\mathcal C(L_+,L_+)(x_+(s),s)}{L_+(x_+(s),s)},
	\]
	\[
	\frac{\mathrm d}{\mathrm ds}\log L_-(x_-(s),s)
	=2b_{-1}(s)+2a_{-1}(s)^2
	\frac{\mathcal C(L_-,L_-)(x_-(s),s)}{L_-(x_-(s),s)}.
	\]
	Since
	\[
	\frac{\mathrm d}{\mathrm ds}\log a_{-1}(s)^2
	=2\bigl(b_0(s)-b_{-1}(s)\bigr),
	\]
	the diagonal terms cancel and
	\[
	\frac{\mathrm d}{\mathrm ds}\log S(s)
	=2a_{-1}(s)^2 \left( \dfrac{\mathcal C(L_-,L_-)(x_-(s),s)}{L_-(x_-(s),s)} - \dfrac{\mathcal C(L_+,L_+)(x_+(s),s)}{L_+(x_+(s),s)} \right)
	\]
	If $L$ is the bilateral Laplace transform of a positive measure, then
	$\mathcal C(L,L)(x)$ has the same sign as $x$.  Here $x_+(s)\ge0$ and
	$x_-(s)\le0$, so $S$ is nonincreasing.  Evaluating at $s=0$ and $s=t$
	gives \eqref{4.2} for $t\ge0$.
	
	\medskip
	\noindent
	\textit{2. Bounds for the boundary transforms.}
	We next prove the local bounds. We shall use
	\begin{equation}
		|\mathcal C(L,L)(x)|\le |x|L(x),
		\qquad
		|L'(0)|\le\frac1r\log\bigl(L(r)+L(-r)\bigr),
		\quad r>0.\label{4.4}
	\end{equation}
	For the first inequality, log-convexity of $L$ implies
	$L(x-y)L(y)\le L(x)$ whenever $y$ lies between $0$ and $x$; integration over
	the oriented segment from $0$ to $x$ proves the claim.  For the second one,
	Jensen's inequality gives
	\[
	\begin{aligned}
		\log\bigl(L(r)+L(-r)\bigr)
		&=\log\int_{\mathbb R}2\cosh(r\lambda)\,\mu(\mathrm d\lambda)\\
		&\ge\int_{\mathbb R}\log\bigl(2\cosh(r\lambda)\bigr)
		\mu(\mathrm d\lambda)\\
		&\ge r\int_{\mathbb R}|\lambda|\,\mu(\mathrm d\lambda)
		\ge r|L'(0)|.
	\end{aligned}
	\]
	
	Introduce the characteristic transforms
	\[
	G_+(x,t)=L_+(x-2t,t),\qquad
	G_-(x,t)=L_-(x+2t,t),
	\]
	and put
	\[
	I_+=[-3T,5T],\qquad I_-=[-5T,3T],
	\]
	\[
	M(t)=\max\left\{
	1,\ \sup_{x\in I_+}G_+(x,t),\
	\sup_{x\in I_-}G_-(x,t)
	\right\}.
	\]
	The asymmetric intervals are chosen so that they contain all characteristic
	preimages of $[-3T,3T]$ during $0\le t\le T$.
	
	Take $r=3T-2t\ge T$ in the second inequality of
	\eqref{4.4}.  Since
	\[
	L_+(r,t)=G_+(3T,t),\qquad
	L_+(-r,t)=G_+(4t-3T,t),
	\]
	and
	\[
	L_-(r,t)=G_-(3T-4t,t),\qquad
	L_-(-r,t)=G_-(-3T,t),
	\]
	all four arguments belong to the corresponding intervals $I_+$ and $I_-$.
	Using $L_+'(0,t)=b_0(t)$ and $L_-'(0,t)=b_{-1}(t)$, we obtain
	\[
	|b_0(t)|,\ |b_{-1}(t)|
	\le\frac1T\log\bigl(2M(t)\bigr).
	\]
	
	Lemma \ref{l14} also gives
	\[
	\partial_tG_+(x,t)=A_+(x,t)G_+(x,t),\qquad
	\partial_tG_-(x,t)=A_-(x,t)G_-(x,t),
	\]
	where
	\[
	A_+(x,t)=-2\left(
	b_0(t)+a_{-1}(t)^2
	\frac{\mathcal C(L_+,L_+)(x-2t,t)}{G_+(x,t)}
	\right),
	\]
	\[
	A_-(x,t)=2\left(
	b_{-1}(t)+a_{-1}(t)^2
	\frac{\mathcal C(L_-,L_-)(x+2t,t)}{G_-(x,t)}
	\right).
	\]
	For $x\in I_+$ or $x\in I_-$ and $0\le t\le T$, the relevant convolution
	argument has absolute value at most $5T$.  Hence
	\eqref{4.4} and the preceding bound for $b_0,b_{-1}$ imply
	\[
	|A_\pm(x,t)|
	\le\frac2T\log\bigl(2M(t)\bigr)+10Ta_{-1}(t)^2.
	\]
	Integrating the logarithmic form of these equations and then taking the
	supremum over $I_+$ and $I_-$ gives
	\[
	\log M(t)
	\le\log M(0)+\int_0^t\left(
	\frac2T\log\bigl(2M(s)\bigr)+10Ta_{-1}(s)^2
	\right)\mathrm ds.
	\]
	The linear Gronwall inequality therefore yields
	\[
	\log M(t)
	\le e^{2t/T}\log M(0)
	+\int_0^t e^{2(t-s)/T}
	\left(\frac2T\log2+10Ta_{-1}(s)^2\right)\mathrm ds,
	\]
	and consequently
	\[
	\log M(t)
	\le e^2\log M(0)+(e^2-1)\log2
	+10Te^2\int_0^Ta_{-1}(s)^2\,\mathrm ds.
	\]
	At $t=0$, the definition of $M$ gives $M(0)\le\mathfrak M(T)$.  Since
	$c_0=10e^2$, the definition of $N_+(T)$ implies $M(t)\le N_+(T)$.
	Moreover, if $|x|\le3T$, then $x+2t\in I_+$ and $x-2t\in I_-$, so
	\[
	L_+(x,t)=G_+(x+2t,t),\qquad
	L_-(x,t)=G_-(x-2t,t).
	\]
	This proves both estimates in \eqref{4.3} on $[0,T]$.
	
	\medskip
	\noindent
	\textit{3. Negative time.}
	For negative time, consider
	\[
	\widehat a_n(s)=a_n(-s),\qquad
	\widehat b_n(s)=-b_n(-s).
	\]
	This is again a Toda solution.  Up to the diagonal gauge
	$u_n\mapsto(-1)^nu_n$, its initial Jacobi operator is $-J_{q(0)}$; hence its
	boundary spectral measures are the pushforwards of the original ones under
	$\lambda\mapsto-\lambda$.  Applying the already proved positive-time result
	to the reflected solution gives \eqref{4.2} and
	\eqref{4.3} on $[-T,0]$, with $N_-(T)$.
	\bigskip
\end{proof}

\begin{proof}[Proof of continuity of $q(t)$ in the metric $d$]
	Let $t_j\to t_0$ and choose $T>|t_0|+1$ so that $|t_j|\leq T$ for all
	large $j$. Lemma \ref{l20}, applied with a sufficiently large time parameter,
	gives
	\[
	\sup_j\int_{\mathbb R}e^{c|\lambda|}
	\sigma_\pm^{q(t_j)}(\mathrm d\lambda)<\infty
	\qquad\text{for every }c>0.
	\]
	Since every coefficient is continuous in time,
	\[
	a_n(t_j)\longrightarrow a_n(t_0),\qquad
	b_n(t_j)\longrightarrow b_n(t_0)
	\qquad(n\in\mathbb Z).
	\]
	Lemma \ref{l7} therefore yields $d(q(t_j),q(t_0))\to0$. \bigskip
\end{proof}

\begin{lemma}\label{l16}
	Let $q_N\in\mathcal Q$ and suppose that each $q_N$ admits a doubly global
	classical Toda solution $q_N(t)$.  If
	\[
	d(q_N,q)\longrightarrow0
	\]
	for some $q\in\mathcal Q$, then $q$ admits a unique doubly global classical
	solution $q(t)$.  Moreover, for every $T>0$,
	\[
	\sup_{|t|\le T}d(q_N(t),q(t))\longrightarrow0.
	\]
\end{lemma}

\begin{proof}
	Write $a_n^{(N)}(t)$ and $b_n^{(N)}(t)$ for the coefficients of $q_N(t)$,
	and let $L_{\pm,N}$ be the corresponding boundary Laplace transforms.  We
	first extract a limit solution on compact time intervals and then prove that
	the whole sequence converges to it.
	
	\medskip
	\noindent
	\textit{1. Compactness at the fixed cut.}
	Fix $T,R>0$ and choose $S>T$ such that $3S>R$.  Since $d(q_N,q)\to0$, the
	numbers $a_{-1}^{(N)}(0)$ are bounded and the initial Laplace transforms
	converge locally uniformly.  The mixed estimate \eqref{4.2} gives
	\[
	\sup_N\sup_{|t|\le S}a_{-1}^{(N)}(t)<\infty.
	\]
	Consequently, the integrals of $(a_{-1}^{(N)})^2$ in $N_\pm(S)$ are uniformly
	bounded. The same is true of the initial exponential moments in
	$\mathfrak M(S)$. Lemma \ref{l20} therefore yields
	\begin{equation}
		\sup_N\sup_{|t|\le T}
		\left(
		p_R(L_{+,N}(\cdot,t))+p_R(L_{-,N}(\cdot,t))
		+|b_{-1}^{(N)}(t)|+|b_0^{(N)}(t)|
		\right)<\infty.\label{4.5}
	\end{equation}
	Indeed, for $|z|\le R$,
	\[
	|L_{\pm,N}(z,t)|
	\le L_{\pm,N}(R,t)+L_{\pm,N}(-R,t).
	\]
	Cauchy's estimate and the moment identities
	\[
	\partial_z^2L_{+,N}(0,t)
	=\bigl(b_0^{(N)}(t)\bigr)^2+\bigl(a_0^{(N)}(t)\bigr)^2,
	\]
	\[
	\partial_z^2L_{-,N}(0,t)
	=\bigl(b_{-1}^{(N)}(t)\bigr)^2+\bigl(a_{-2}^{(N)}(t)\bigr)^2
	\]
	then show that $a_{-2}^{(N)}$ and $a_0^{(N)}$ are also uniformly bounded on
	compact time intervals.
	
	For $n\in\mathbb Z$, let $F_{n,N}=F_{n}^{q_N}$ be the entire function from
	Theorem \ref{p26}.  On the real axis,
	\[
	F_{n,N}(t)=\exp\left(
	4\int_0^t(t-s)\bigl(a_n^{(N)}(s)\bigr)^2\,\mathrm ds
	\right).
	\]
	By \eqref{4.5} and the preceding bounds,
	$F_{n,N}(R)+F_{n,N}(-R)$ is bounded uniformly in $N$ for
	$n=-2,-1,0$ and every $R>0$.  Since $F_{n,N}$ is a bilateral Laplace
	transform,
	\[
	\sup_{|z|\le R}|F_{n,N}(z)|
	\le F_{n,N}(R)+F_{n,N}(-R).
	\]
	Thus these three families are normal.  A diagonal application of Montel's
	theorem gives a subsequence, still denoted by $N$, such that
	\begin{equation}
		F_{n,N}\longrightarrow F_n
		\quad\text{locally uniformly on }\mathbb C,
		\qquad n=-2,-1,0.\label{4.6}
	\end{equation}
	
	The Hankel positivity proved in \eqref{3.9} passes to the limit.  Lemma \ref{l11}
	therefore gives probability measures $\rho_n$ with exponential moments of
	every order such that
	\[
	F_n(z)=\int_{\mathbb R}e^{2\lambda z}\rho_n(\mathrm d\lambda),
	\qquad n=-2,-1,0.
	\]
	For real $t$, define the limiting coefficients directly by
	\begin{equation}
		a_n(t)^2=\frac14(\log F_n)''(t),
		\qquad n=-2,-1,0,\label{4.7}
	\end{equation}
	and take $a_n(t)$ to be the positive square root.  The right-hand side of
	\eqref{4.7} is the variance of $\lambda$ with respect to the tilted
	probability measure proportional to $e^{2t\lambda}\rho_n(\mathrm d\lambda)$;
	hence it is nonnegative.  If it vanished at one real point, then $\rho_n$
	would be a point mass and the right-hand side would vanish identically.  This
	is impossible because \eqref{4.6} and Cauchy's formula give
	\[
	a_n(0)^2
	=\lim_{N\to\infty}\bigl(a_n^{(N)}(0)\bigr)^2
	=a_n(q)^2>0.
	\]
	Thus $a_n(t)>0$ for all real $t$ and $n=-2,-1,0$.
	
	On a complex neighborhood of every compact real interval, $F_n$ has no
	zeros; by \eqref{4.6}, the same is true of $F_{n,N}$ for all large
	$N$.  Cauchy's formula then yields local uniform convergence, together with
	all derivatives, of
	\[
	\frac14(\log F_{n,N})''=\bigl(a_n^{(N)}\bigr)^2
	\]
	to the function $a_n^2$ in \eqref{4.7}.  Since all coefficients are
	positive on the real axis,
	\begin{equation}
		a_n^{(N)}\longrightarrow a_n,
		\qquad n=-2,-1,0,\label{4.8}
	\end{equation}
	uniformly on compact time intervals.
	
	Using the integral form of the second Toda equation, define
	\[
	b_{-1}(t)=b_{-1}(q)
	+2\int_0^t\bigl(a_{-1}(s)^2-a_{-2}(s)^2\bigr)\,\mathrm ds,
	\]
	\[
	b_0(t)=b_0(q)
	+2\int_0^t\bigl(a_0(s)^2-a_{-1}(s)^2\bigr)\,\mathrm ds.
	\]
	Lemma \ref{l7} and \eqref{4.8} imply that
	$b_{-1}^{(N)}\to b_{-1}$ and $b_0^{(N)}\to b_0$ uniformly on compact time
	intervals. Moreover, because the derivatives of these functions are fixed
	linear combinations of the holomorphic functions $(a_n^{(N)})^2$,
	$n=-2,-1,0$, the
	convergence holds locally uniformly, together with all derivatives, on a
	complex neighborhood of every compact real interval.
	
	\medskip
	\noindent
	\textit{2. Recursive construction of all coefficients.}
	We first move to the right.  Suppose that, for some $n\ge0$,
	$(a_n^{(N)})^2\to a_n^2$ and $b_n^{(N)}\to b_n$ locally uniformly, together
	with all derivatives, near every compact real interval, and that
	$a_n(t)^2>0$ on $\mathbb R$. On a common neighborhood where
	$(a_n^{(N)})^2$ and $a_n^2$ do not vanish, the Toda equations give the first
	line below, and we define $b_{n+1}$ and $a_{n+1}^2$ by the right-hand sides
	in the second line:
	\begin{equation}
		\begin{aligned}
			b_{n+1}^{(N)}
			&=b_n^{(N)}
			+\frac12\frac{\bigl((a_n^{(N)})^2\bigr)'}{(a_n^{(N)})^2},
			&
			\bigl(a_{n+1}^{(N)}\bigr)^2
			&=\bigl(a_n^{(N)}\bigr)^2
			+\frac12\bigl(b_{n+1}^{(N)}\bigr)',\\
			b_{n+1}
			&=b_n+\frac12\frac{(a_n^2)'}{a_n^2},
			&
			a_{n+1}^2
			&=a_n^2+\frac12b_{n+1}'.
		\end{aligned}\label{4.9}
	\end{equation}
	Hence the first line converges locally uniformly, together with all
	derivatives, to the second line.  In particular,
	$a_{n+1}(0)^2=a_{n+1}(q)^2>0$.
	
	The resulting bounds make $\{F_{n+1,N}\}_N$ a normal family by the same
	real-axis estimate as before.  Every cluster point has, on the real axis,
	the restriction
	\[
	\exp\left(
	4\int_0^t(t-s)a_{n+1}(s)^2\,\mathrm ds
	\right).
	\]
	The identity theorem makes the cluster point unique; hence $F_{n+1,N}$
	converges locally uniformly to an entire function $F_{n+1}$.  Passing Hankel
	positivity to the limit and repeating the variance argument used in
	\eqref{4.7} shows that
	\[
	a_{n+1}(t)^2=\frac14(\log F_{n+1})''(t)>0,
	\qquad t\in\mathbb R.
	\]
	Thus \eqref{4.9}, starting with $n=0$, constructs all
	coefficients to the right.
	
	To move to the left, suppose that $a_n^2$ and $b_{n+1}$ have already been
	constructed.  The corresponding identities are
	\begin{equation}
		\begin{aligned}
			b_n^{(N)}
			&=b_{n+1}^{(N)}
			-\frac12\frac{\bigl((a_n^{(N)})^2\bigr)'}{(a_n^{(N)})^2},
			&
			\bigl(a_{n-1}^{(N)}\bigr)^2
			&=\bigl(a_n^{(N)}\bigr)^2
			-\frac12\bigl(b_n^{(N)}\bigr)',\\
			b_n
			&=b_{n+1}-\frac12\frac{(a_n^2)'}{a_n^2},
			&
			a_{n-1}^2
			&=a_n^2-\frac12b_n'.
		\end{aligned}\label{4.10}
	\end{equation}
	Starting with $n=-2$ and applying the same normal-family, uniqueness, and
	variance arguments constructs all coefficients to the left.  Therefore,
	along the selected subsequence, for every fixed $n\in\mathbb Z$,
	\begin{equation}
		a_n^{(N)}\longrightarrow a_n,
		\qquad
		b_n^{(N)}\longrightarrow b_n \label{4.11}
	\end{equation}
	uniformly on compact time intervals, and $a_n(t)>0$ for all $n$ and $t$.
	
	\medskip
	\noindent
	\textit{3. Passage to the limit and convergence in $d$.}
	For every fixed $n$, the approximating solutions satisfy
	\begin{equation}
		\begin{aligned}
			a_n^{(N)}(t)
			&=a_n^{(N)}(0)+\int_0^t a_n^{(N)}(s)
			\bigl(b_{n+1}^{(N)}(s)-b_n^{(N)}(s)\bigr)\,\mathrm ds,\\
			b_n^{(N)}(t)
			&=b_n^{(N)}(0)+2\int_0^t
			\left(\bigl(a_n^{(N)}(s)\bigr)^2
			-\bigl(a_{n-1}^{(N)}(s)\bigr)^2\right)\,\mathrm ds.
		\end{aligned} \label{4.12}
	\end{equation}
	Passing to the limit in \eqref{4.12} by
	\eqref{4.11} shows that the limiting coefficients form a
	doubly global classical Toda solution with initial value $q$.
	The results of Section 3 imply that this solution lies in $\mathcal Q$ at
	every time and is unique.
	
	It remains to prove convergence in the metric. Fix $T>0$. Lemma \ref{l20},
	applied with a sufficiently large time parameter, gives for every $A>0$
	\[
	\sup_N\sup_{|t|\le T}
	\int_{\mathbb R}e^{A|\lambda|}
	\sigma_\pm^{q_N(t)}(\mathrm d\lambda)<\infty.
	\]
	We claim that, along the selected subsequence,
	\begin{equation}
		\sup_{|t|\le T}d(q_N(t),q(t))\longrightarrow 0.\label{4.13}
	\end{equation}
	Suppose otherwise. Then there are $\varepsilon_0>0$, a subsequence $N_j$,
	and $t_j\in[-T,T]$ such that
	\[
	d(q_{N_j}(t_j),q(t_j))\geq\varepsilon_0.
	\]
	After passing to a further subsequence, assume that $t_j\to t_0$. The
	coefficient convergence in \eqref{4.11}, uniform on $[-T,T]$, implies
	\[
	a_n^{(N_j)}(t_j)\longrightarrow a_n(t_0),\qquad
	b_n^{(N_j)}(t_j)\longrightarrow b_n(t_0)
	\qquad(n\in\mathbb Z).
	\]
	The preceding uniform exponential estimate and Lemma \ref{l7} therefore give
	\[
	d(q_{N_j}(t_j),q(t_0))\longrightarrow0.
	\]
	On the other hand, continuity of the limiting trajectory gives
	$d(q(t_j),q(t_0))\to0$, a contradiction.
	Thus \eqref{4.13} holds.
	
	Finally, the preceding argument applies to every subsequence of the original
	sequence, while uniqueness forces every subsequential limit to be $q(t)$.
	If the whole sequence failed to satisfy \eqref{4.13}, one could
	choose a subsequence staying a fixed positive distance from $q(t)$ on
	$[-T,T]$; that subsequence would have a further subsequence satisfying
	\eqref{4.13}, a contradiction.  Thus the whole sequence converges,
	and the proof is complete.
	\bigskip
\end{proof}

Choose a probability measure $\nu$ supported on $[-1,1]$ and having infinite
support. For $q \in \mathcal{Q}$ and $N\ge1$, define
\[
\sigma_{\pm,N}(\mathrm d\lambda)
=
\frac{
	\mathbf 1_{[-N,N]}(\lambda)\sigma_\pm^q(\mathrm d\lambda)
	+N^{-1}\nu(\mathrm d\lambda)
}{
	\sigma_\pm^q([-N,N])+N^{-1}
}.
\]
The additional term ensures that every $\sigma_{\pm,N}$ has infinite support. We join them by
setting $a_{-1}^{(N)}=a_{-1}(q)$ and denote the resulting two-sided Jacobi coefficients from $\sigma_{\pm,N}$ by
\[
q^{(N)}=\{a_n^{(N)},b_n^{(N)}\}_{n\in\mathbb Z}.
\]

\begin{lemma}\label{l18}
	For every $c>0$,
	\[
	\sup_{N\ge1}
	\int_{\mathbb R}e^{c|\lambda|}\sigma_{\pm,N}(\mathrm d\lambda)<\infty,
	\]
	and
	\[
	\mathcal L_{\sigma_{\pm,N}}
	\longrightarrow
	\mathcal L_{\sigma_\pm^q}
	\]
	locally uniformly on $\mathbb C$, and then $d(q^{(N)},q) \to 0$.  In particular, all moments of
	$\sigma_{\pm,N}$ converge to the corresponding moments of
	$\sigma_\pm^q$.
\end{lemma}

\begin{proof}
	We omit the proof since it is routine. \bigskip
\end{proof}

\begin{proof}[Proof of global existence and continuity]
	For each $N$, the standard bounded Toda theory gives a unique doubly global
	bounded solution $q^{(N)}(t)$ with initial value $q^{(N)}$; see
	\cite[Chapter 12]{GT}.  Lemma \ref{l16} and \ref{l18} imply that
	$q$ admits a doubly global classical solution $q(t)$ and that, for every
	$T>0$,
	\[
	\sup_{|t|\le T}d(q^{(N)}(t),q(t))\longrightarrow0.
	\]
	\bigskip
\end{proof}

Combining the proof of global existence and continuity in this section and the proof of exponential integrability and uniqueness in section 3, we complete the proof of Theorem 1.

\end{document}